\documentclass[12pt]{amsart}

\usepackage{setspace}
\RequirePackage{amssymb}
\usepackage{fullpage}
\usepackage{amsbsy}
\usepackage{enumerate}
\usepackage{mathrsfs}
\usepackage{amsmath}
\usepackage{url}
\usepackage{comment}

\theoremstyle{plain}
\newtheorem{thm}{Theorem}[section]
\newtheorem*{thm*}{Theorem}
\newtheorem*{mainthm*}{Main Theorem}
\newtheorem*{mainlem*}{Main Lemma}
\newtheorem{lem}[thm]{Lemma} \newtheorem*{lem*}{Lemma}
\newtheorem{claim}[thm]{Claim} \newtheorem*{claim*}{Claim}
 \newtheorem*{cor*}{Corollary}
\newtheorem{prop}[thm]{Proposition} \newtheorem*{prop*}{Proposition}

\theoremstyle{definition}
 \newtheorem*{defn*}{Definition}

\theoremstyle{remark}
\newtheorem{rem}[thm]{Remark} \newtheorem*{rem*}{Remark}
 \newtheorem*{example*}{Example}
 \newtheorem*{conj*}{Conjecture}
 \newtheorem*{question*}{Question}

\newcommand{\Ord}{\mathrm{Ord}}

\DeclareMathOperator{\powerset}{\mathcal{P}}

\DeclareMathOperator{\rank}{rank}

\begin{document}

\author{Trevor M.\ Wilson}
\title{Weak Vop\v{e}nka's Principle does not imply Vop\v{e}nka's Principle}

\address{Department of Mathematics\\Miami University\\Oxford, Ohio 45056\\ USA}
\email{twilson@miamioh.edu} 


\begin{abstract}
 Vop\v{e}nka's Principle says that the category of graphs has no large discrete full subcategory, or equivalently that the category of ordinals cannot be fully embedded into it. Weak Vop\v{e}nka's Principle is the dual statement, which says that the opposite category of ordinals cannot be fully embedded into the category of graphs. It was introduced in 1988 by Ad\'{a}mek, Rosick\'{y}, and Trnkov\'{a}, who showed that it follows from Vop\v{e}nka's Principle and asked whether the two statements are equivalent. We show that they are not.  However, we show that Weak Vop\v{e}nka's Principle is equivalent to the generalization of itself known as Semi-Weak Vop\v{e}nka's Principle.
\end{abstract}

\maketitle

\section{Introduction}

A \emph{graph} is a structure $\langle G, E\rangle$ where $G$ is a set and $E$ is a binary relation on $G$. (Sometimes we may denote the structure itself by $G$.) A \emph{homomorphism} from $\langle G, E\rangle$ to $\langle G', E'\rangle$ is a function $h : G \to G'$ such that $\langle v,w\rangle \in E$ implies $\langle h(v),h(w)\rangle \in E'$.

Despite its simple definition, the category of graphs is sufficiently complex that some aspects of its behavior are independent of the usual axioms of set theory, leading us to consider additional principles beyond those axioms. One such principle is \emph{Vop\v{e}nka's Principle} (VP), which states that the category of graphs has no large discrete full subcategory. In other words, for every proper class of graphs there is a non-identity homomorphism among the graphs in that class.

The distinction between sets and proper classes is crucial to VP because a result of Vop\v{e}nka, Pultr, and Hedrl\'{\i}n \cite{VopPulHedRigidRelation} implies that for every cardinal $\kappa$ there is a set of $\kappa$ many graphs without any non-identity homomorphisms among its members. 

Vop\v{e}nka's Principle is not actually specific to graphs because many other categories can be fully embedded into the category of graphs. For example, every locally presentable category can be so embedded (see Ad\'{a}mek and Rosick\'{y} \cite[Theorem 2.65]{AdaRosLocallyPresentable}.) Vop\v{e}nka's Principle and the weak forms of it discussed below therefore admit equivalent statements in terms of locally presentable categories.

Ad\'{a}mek, Rosick\'{y}, and Trnkov\'{a} \cite[Lemma 1]{AdaRosTrnLimitClosed} observed that Vop\v{e}nka's principle is equivalent to the statement that $\text{Ord}$ (the well-ordered category of all ordinals) cannot be fully embedded into the category of graphs, meaning that no sequence of graphs $\langle G_\alpha : \alpha \in \Ord\rangle$ has both of the following properties: whenever $\alpha \le \alpha'$ there is a unique homomorphism $G_\alpha \to G_{\alpha'}$ and whenever $\alpha < \alpha'$ there is no homomorphism $G_{\alpha'} \to G_{\alpha}$.

They defined \emph{Weak Vop\v{e}nka's Principle} (WVP) as the dual statement that the opposite category $\text{Ord}^\text{op}$ cannot be fully embedded into the category of graphs, meaning that no sequence of graphs $\langle G_\alpha : \alpha \in \Ord\rangle$ has both of the following properties: whenever $\alpha \le \alpha'$ there is a unique homomorphism $G_{\alpha'} \to G_{\alpha}$ and whenever $\alpha < \alpha'$ there is no homomorphism $G_\alpha \to G_{\alpha'}$.  They showed that Vop\v{e}nka's Principle implies WVP \cite[Lemma 2]{AdaRosTrnLimitClosed} and asked whether the two principles are equivalent.

We will show that WVP is not equivalent to Vop\v{e}nka's Principle. As an immediate consequence of this result, the statement ``in every locally presentable category, every full subcategory closed under limits is reflective\footnote{A subcategory $\mathscr{A}$ of a category $\mathscr{K}$ is called \emph{reflective} if for every object $K \in \mathscr{K}$ there is a morphism $K \to K^* \in \mathscr{A}$ through which every morphism $K \to A \in \mathscr{A}$ factors uniquely.}'' is not equivalent to the dual statement ``in every locally presentable category, every full subcategory closed under colimits is coreflective'' because the two statements are equivalent to WVP and VP respectively \cite{AdaRosTrnLimitClosed, RosTrnAdaUnexpected}.

Uniqueness of homomorphisms can be removed from the statement of WVP to obtain \emph{Semi-Weak Vop\v{e}nka's Principle} (SWVP), introduced by Ad\'{a}mek and Rosick\'{y} \cite{AdaRosInjectivity}, which says that no sequence of graphs $\langle G_\alpha : \alpha \in \Ord\rangle$ has both of the following properties: whenever $\alpha \le \alpha'$ there is a homomorphism $G_{\alpha'} \to G_{\alpha}$ (not necessarily unique) and whenever $\alpha < \alpha'$ there is no homomorphism $G_\alpha \to G_{\alpha'}$.

We will show that WVP and SWVP are equivalent. As an immediate consequence of this result, the statement ``in every locally presentable category, every full subcategory closed under limits is reflective'' is equivalent to the statement ``in every locally presentable category, every full subcategory closed under products and retracts is weakly reflective\footnote{A subcategory $\mathscr{A}$ of a category $\mathscr{K}$ is called \emph{weakly reflective} if it is closed under retracts and for every object $K \in \mathscr{K}$ there is a morphism $K \to K^* \in \mathscr{A}$ through which every morphism $K \to A \in \mathscr{A}$ factors, not necessarily uniquely.}'' because the latter statement is equivalent to SWVP \cite[Corollary I.15]{AdaRosInjectivity}.

Vop\v{e}nka's Principle implies Semi-Weak Vop\v{e}nka's Principle \cite{AdaRosInjectivity} and it is clear from the definitions that Semi-Weak Vop\v{e}nka's Principle implies Weak Vop\v{e}nka's Principle. We will show that only the second implication can be reversed:
\[ \text{VP} \underset{\not \Leftarrow}\implies \text{SWVP} \iff \text{WVP}.\]

We will also obtain a result concerning strongly compact cardinals, whose definition we now review. For a cardinal $\kappa$ and a set $X$, we let $\powerset_{\kappa}(X)= \{ A \subset X : |A| < \kappa\}$, meaning the set of all subsets of $X$ of cardinality less than $\kappa$. For an uncountable cardinal $\kappa$ and a set $X$ we say that $\kappa$ is \emph{$X$-compact} if there is a fine $\kappa$-complete ultrafilter on $\powerset_{\kappa}(X)$, where \emph{fine} means that for every $x \in X$ the set $\{A \in \powerset_{\kappa}(X) : x \in A\}$ is in the ultrafilter and \emph{$\kappa$-complete} means that the intersection of fewer than $\kappa$ many sets in the ultrafilter is in the ultrafilter. An uncountable cardinal $\kappa$ is called \emph{strongly compact}
if it is $X$-compact for every set $X$.

We will show that SWVP (and therefore also WVP) does not imply the existence of a strongly compact cardinal, answering a question raised by Ad\'{a}mek and Rosick\'{y} \cite{AdaRosInjectivity}.

A technical augmentation of strong compactness known as supercompactness will be useful here. For an uncountable cardinal $\kappa$ and a set $X$ we say that $\kappa$ is \emph{$X$-supercompact} if there is a \emph{normal} fine $\kappa$-complete ultrafilter on $\powerset_{\kappa}(X)$, where normality means that every choice function\footnote{A \emph{choice function} is a function $f$ such that $f(A) \in A$ for every set $A$ in the domain of $f$.} defined on a set in the ultrafilter is constant on a set in the ultrafilter. An uncountable cardinal $\kappa$ is called \emph{supercompact} if it is $X$-supercompact for every set $X$.

We will separate SWVP from Vop\v{e}nka's Principle by showing that the two principles relate to supercompactness differently. Vop\v{e}nka's Principle implies that there is a supercompact cardinal by Solovay, Reinhardt, and Kanamori \cite[Theorem 6.9]{SolReiKanStrongAxioms}, and in fact Bagaria \cite[Corollary 4.6]{BagariaCNCardinals} showed that the existence of a supercompact cardinal is equivalent to a fragment of Vop\v{e}nka's Principle for certain definable classes. By contrast, we will show that the existence of a supercompact cardinal implies the existence of a set-sized \emph{model} of SWVP.

The natural setting for Vop\v{e}nka's Principle and its weak forms is second-order set theory, which allows reference to arbitrary classes and arbitrary Ord-sequences of graphs without any assumption of definability,
and the natural models of second-order set theory are the structures $\langle V_{\kappa+1}, \in \rangle$ where $\kappa$ is an inaccessible cardinal. Here $V_\alpha$ denotes the $\alpha^\text{th}$ level of the cumulative hierarchy defined by $V_0 = \emptyset$, $V_{\alpha+1} = \powerset(V_\alpha)$, and $V_\lambda = \bigcup_{\alpha < \lambda} V_\alpha$ for every limit ordinal $\lambda$. In other words, $V_\alpha$ is the set of all sets of rank less than $\alpha$. The structure $\langle V_{\kappa+1}, \in \rangle$ sees elements of $V_\kappa$ and $V_{\kappa+1}$ as its sets and classes respectively.\footnote{We don't need to explicitly add the type distinction between sets and classes to the structure, because the set-hood property ``$x \in V_\kappa$'' can be defined in the structure by the formula ``$x \in y$ for some $y$.''}

The following proposition gives a sufficient\footnote{But not necessary: see Remark \ref{rem:Woodin}.} condition to obtain a model of SWVP:

\begin{prop}[ZFC] \label{prop:V-kappa-satisfies-SWVP}
 If $\kappa$ is a supercompact cardinal, then the structure $\langle V_{\kappa+1}, \in\rangle$ satisfies Semi-Weak Vop\v{e}nka's Principle.
\end{prop}

This result will allow us to easily separate SWVP from the existence of supercompact and strongly compact cardinals under appropriate minimality assumptions on $\kappa$.

\begin{thm}[ZFC] \label{thm:V-kappa-satisfies-SWVP-and-not-exists-spct}
 If $\kappa$ is the first supercompact cardinal, then $\langle V_{\kappa+1}, \in\rangle$ satisfies Semi-Weak Vop\v{e}nka's Principle 
 and the statement ``there is no supercompact cardinal.''
\end{thm}

It follows from Theorem \ref{thm:V-kappa-satisfies-SWVP-and-not-exists-spct} that SWVP does not provably imply VP (unless it is inconsistent) because the latter principle does imply the existence of a supercompact cardinal.

\begin{thm}[ZFC] \label{thm:V-kappa-satisfies-SWVP-and-not-exists-SC}
 If $\kappa$ is the first supercompact cardinal and the first strongly compact cardinal, then $\langle V_{\kappa+1}, \in\rangle$ satisfies Semi-Weak Vop\v{e}nka's Principle and the statement ``there is no strongly compact cardinal.''
\end{thm}

Magidor \cite{MagIdentityCrisis} showed that the hypothesis of Theorem \ref{thm:V-kappa-satisfies-SWVP-and-not-exists-SC} is consistent, assuming that the existence of a supercompact cardinal is consistent. It follows that Semi-Weak Vop\v{e}nka's Principle does not provably imply the existence of a strongly compact cardinal (unless the existence of a supercompact cardinal is inconsistent.)

Finally, we will show that the weak and semi-weak forms of VP are equivalent:

\begin{thm}[GB + AC]\label{thm:wvp-implies-swvp}
 Weak Vop\v{e}nka's Principle is equivalent to Semi-Weak Vop\v{e}nka's Principle.
\end{thm}

For maximum generality we prove Theorem \ref{thm:wvp-implies-swvp} in the second-order set theory GB + AC, which is G\"odel--Bernays set theory with the Axiom of Choice for sets.  Every model of ZFC together with its definable classes forms a model of GB + AC, so the result also holds for definable Ord-sequences of graphs in ZFC as a special case.

Proposition \ref{prop:V-kappa-satisfies-SWVP} and Theorems \ref{thm:V-kappa-satisfies-SWVP-and-not-exists-spct} and \ref{thm:V-kappa-satisfies-SWVP-and-not-exists-SC} will be proved in Section \ref{sec:vp-is-stronger}.  Theorem \ref{thm:wvp-implies-swvp} will be proved in Section \ref{sec:wvp-implies-swvp}. The author is grateful to Joan Bagaria for his helpful comments on an earlier version of this article.

\section{Models of Semi-Weak Vop\v{e}nka's Principle}\label{sec:vp-is-stronger}

In this section we obtain models of Semi-Weak Vop\v{e}nka's Principle from supercompact cardinals (Proposition \ref{prop:V-kappa-satisfies-SWVP}.) Then by adding appropriate minimality assumptions to the hypothesis of supercompactness, we obtain models separating SWVP from the existence of supercompact and strongly compact cardinals (Theorems \ref{thm:V-kappa-satisfies-SWVP-and-not-exists-spct} and \ref{thm:V-kappa-satisfies-SWVP-and-not-exists-SC} respectively.)

Our key application of supercompactness is the following lemma regarding homomorphisms of product graphs. In the set-theoretic construction of the product, an element of $\prod_{i \in I} G_i$ is a function $f$ with domain $I$ such that $f(i) \in G_i$ for all $i \in I$, and $f$ is adjacent to $g$ in $\prod_{i \in I} G_i$ if and only if $f(i)$ is adjacent to $g(i)$ in $G_i$ for all $i \in I$. (When we say that a vertex $v$ is adjacent to a vertex $w$ in a graph $G$, we mean that the ordered pair $\langle v,w\rangle$ is an element of the edge relation of $G$, which we do not name explicitly.)

\begin{lem}[ZFC]\label{lem:product-hom}
 Let $\kappa$ be a supercompact cardinal, let $I$ be a set of cardinality $\kappa$, and for each index $i \in I$ let $G_i$ be a graph of cardinality less than $\kappa$. Then there is a set  $I_0 \subset I$ of cardinality less than $\kappa$ and a homomorphism from $\prod_{i \in I_0} G_i$ to $\prod_{i \in I} G_i$.
\end{lem}
\begin{proof}
 It will suffice to show that there is a set $A \subset I \cup \prod_{i \in I} G_i$ of cardinality less than $\kappa$ such that the restriction function 
 \[\rho_A : A \cap \prod_{i \in I} G_i \to \prod_{i \in A \cap I} G_i\]
 defined by $\rho_A(f) = f \restriction (A \cap I)$ is a bijection and preserves nonadjacency. Then for $I_0 = A \cap I$, the inverse of $\rho_A$ will be the desired homomorphism from $\prod_{i \in I_0} G_i$ to  $\prod_{i \in I} G_i$.
 
 In order to show that some set $A \in \powerset_\kappa(I \cup \prod_{i \in I} G_i)$ has this property, we will show the apparently stronger fact that $U$-almost every set $A$ has this property where $U$ is a normal fine $\kappa$-complete ultrafilter on $\powerset_\kappa(I \cup \prod_{i \in I} G_i)$.\footnote{The lemma can be proved by working with the ultrafilter or with the corresponding ultrapower of $\langle V, \in\rangle$.  Because the ultrapower construction is not really necessary here, we will work directly with the ultrafilter.} Such a ultrafilter exists because $\kappa$ is supercompact. When we say that \emph{$U$-almost every} set $A$ has a property, we mean that the set of all sets $A \in \powerset_\kappa(I \cup \prod_{i \in I} G_i)$ with the property is an element of $U$.
 
 First we show that $\rho_A$ is injective for $U$-almost every set $A$. If not, then for $U$-almost every set $A$ we may choose distinct $f_A, g_A \in A \cap \prod_{i \in I} G_i$ whose restrictions are equal: $f_A \restriction (A \cap I) = g_A \restriction (A \cap I)$. We may eliminate the dependence on $A$ by applying normality to the choice functions $A \mapsto f_A$ and $A \mapsto g_A$, obtaining
 distinct
 $f,g \in \prod_{i \in I} G_i$
 such that $f \restriction (A \cap I) = g \restriction (A \cap I)$ for $U$-almost every set $A$. This contradicts fineness of $U$: any index $i \in I$ on which $f$ and $g$ differ must be included in $A \cap I$ for $U$-almost every set $A$.
 
 Next we use a similar argument to show that $\rho_A$ preserves nonadjacency for $U$-almost every set $A$. If not, then for $U$-almost every set $A$ we may choose $f_A, g_A \in A \cap \prod_{i \in I} G_i$ such that $f_A$ is not adjacent to $g_A$ in $\prod_{i \in I} G_i$ but the restriction $f_A \restriction (A \cap I)$ is adjacent to $g_A \restriction (A \cap I)$ in the smaller product $\prod_{i \in A \cap I} G_i$.
 Then by normality we obtain $f,g \in \prod_{i \in I} G_i$ such that $f$ is not adjacent to $g$ in $\prod_{i \in I} G_i$ but the restriction $f \restriction (A \cap I)$ is adjacent to $g \restriction (A \cap I)$ in $\prod_{i \in A \cap I} G_i$ for $U$-almost every set $A$. This contradicts fineness of $U$: any index $i \in I$ on which $f$ is not adjacent to $g$
 (meaning such that $f(i)$ is not adjacent to $g(i)$ in $G_i$)
 must be included in $A \cap I$ for $U$-almost every set $A$.
 
 Finally we use a different type of argument to show that $\rho_A$ is surjective for $U$-almost every set $A$. If not, then for $U$-almost every set $A$ we may choose $f_A \in \prod_{i \in A \cap I} G_i$ that is not in the range of $\rho_A$, meaning that no  element of the larger product $\prod_{i \in I} G_i$ extending $f_A$ is in $A$. In this case we cannot use normality to fix $f_A$ on a set in $U$ because $f_A$ is not an element of $A$. However, we can use the ultrafilter $U$ to ``integrate'' the various choices of $f_A$ to obtain $f \in  \prod_{i \in I} G_i$.
 
 Namely, we can define $f \in \prod_{i \in I} G_i$ by $f(i) = v$ if and only if $f_A(i) = v$ for $U$-almost every set $A$. We must verify two conditions to show that this definition makes sense. First, because $U$ is fine, for $U$-almost every set $A$ we have $i \in A \cap I$, so $f_A(i)$ is defined and is an element of $G_i$. Second, because $U$ is a $\kappa$-complete ultrafilter and $G_i$ has cardinality less than $\kappa$, there is a unique choice of $v \in G_i$ such that $f_A(i) = v$ for $U$-almost every set $A$.
 
 Now by fineness of $U$ we have $f \in A \cap \prod_{i \in I} G_i$ for $U$-almost every set $A$.  By our choice of $f_A$, we must have $f \restriction (A \cap I) \ne f_A$ for $U$-almost every set $A$. For all such $A$ we may choose an index $i_A \in A \cap I$ on which they differ: $f(i_A) \ne f_A(i_A)$. Then by normality of $U$ we may fix $i \in I$ such that $i_A = i$ for $U$-almost every set $A$. For this index $i$ we have $f(i) \ne f_A(i)$ for $U$-almost every set $A$, contradicting the definition of $f(i)$.
\end{proof}

Proposition \ref{prop:V-kappa-satisfies-SWVP} will now follow by a straightforward observation on the relationship between SWVP and homomorphisms of product graphs.

\begin{proof}[Proof of Proposition \ref{prop:V-kappa-satisfies-SWVP}]
 Assume ZFC and let $\kappa$ be a supercompact cardinal. Then $\kappa$ is inaccessible, so $\langle V_{\kappa+1}, \in\rangle$ satisfies GB + AC. We want to show that it also satisfies Semi-Weak Vop\v{e}nka's Principle. Let $\langle G_\alpha : \alpha < \kappa\rangle$ be a sequence of graphs in $V_{\kappa+1}$ (so each $G_\alpha$ is in $V_\kappa$) such that for all ordinals $\alpha \le \alpha' < \kappa$ there is a homomorphism from $G_{\alpha'}$ to $G_\alpha$. We want to show that for some $\alpha < \alpha' < \kappa$ there is a homomorphism from $G_\alpha$ to $G_{\alpha'}$.  In particular, we will show that for some $\beta < \kappa$ there is a homomorphism from $G_\beta$ to $G_{\beta+1}$.
 
 By Lemma \ref{lem:product-hom} there is a homomorphism $\prod_{\alpha \in I_0} G_\alpha \to \prod_{\alpha <\kappa} G_\alpha$ for some set $I_0 \subset \kappa$ of cardinality less than $\kappa$. Let $\beta$ be the supremum of $I_0$, which is less than $\kappa$ because $\kappa$ is a regular cardinal. Then by our hypothesis on the existence of ``backward'' homomorphisms we may choose homomorphisms $G_\beta \to G_\alpha$ for all $\alpha \in I_0$ and combine them to obtain a homomorphism $G_\beta \to \prod_{\alpha \in I_0} G_\alpha$. Moreover, we have a projection homomorphism $\prod_{\alpha <\kappa} G_\alpha \to G_{\beta+1}$. We may compose these three homomorphisms
 \[ G_\beta \to \prod_{\alpha \in I_0} G_\alpha \to \prod_{\alpha <\kappa} G_\alpha \to G_{\beta+1}\]
 to obtain the desired homomorphism from $G_\beta$ to $G_{\beta+1}$.
\end{proof}

Theorems \ref{thm:V-kappa-satisfies-SWVP-and-not-exists-spct} and \ref{thm:V-kappa-satisfies-SWVP-and-not-exists-SC} will now follow easily.

\begin{proof}[Proof of Theorem \ref{thm:V-kappa-satisfies-SWVP-and-not-exists-spct}]
 Let $\kappa$ be the first supercompact cardinal. Then Semi-Weak Vop\v{e}nka's Principle holds in $\langle V_{\kappa+1}, \in \rangle$ by Proposition \ref{prop:V-kappa-satisfies-SWVP}. Suppose toward a contradiction that the statement ``there is a supercompact cardinal'' holds in $\langle V_{\kappa+1}, \in \rangle$. This statement is first-order, so we might as well just say that it holds in the structure $\langle V_{\kappa}, \in \rangle$. Let $\kappa_0 < \kappa$ be a witness to this statement, meaning that $\kappa_0$ is $X$-supercompact for every set $X \in V_\kappa$. Because $\kappa$ itself is supercompact, this implies that $\kappa_0$ is $X$-supercompact for every set $X \in V$, \emph{i.e.}\ it is supercompact: see Kanamori \cite[Exercise 22.9]{KanHigherInfinite}.  This contradicts the minimality of $\kappa$.
\end{proof}

\begin{proof}[Proof of Theorem \ref{thm:V-kappa-satisfies-SWVP-and-not-exists-SC}]
 The same as the proof of Theorem \ref{thm:V-kappa-satisfies-SWVP-and-not-exists-spct}, using the fact that if $\kappa_0$ is $X$-compact for every set $X \in V_\kappa$ where $\kappa$ is supercompact, then $\kappa_0$ is strongly compact.
\end{proof}

\begin{rem}\label{rem:Woodin}
 The proof of Lemma \ref{lem:product-hom} used only $2^\kappa$-supercompactness of $\kappa$, meaning $X$-supercompactness for a set $X$ of cardinality at most $2^\kappa$. This is the smallest fragment of supercompactness that suffices for Lemma \ref{lem:product-hom} and Proposition \ref{prop:V-kappa-satisfies-SWVP}, but it is not the weakest large cardinal hypothesis that suffices. As we will show in another article, the weakest large cardinal hypothesis that suffices for these results is ``$\kappa$ is a Woodin cardinal.''
\end{rem}

\section{Weak Vop\v{e}nka's Princple implies Semi-Weak Vop\v{e}nka's Princple}\label{sec:wvp-implies-swvp}

In this section we will prove Theorem \ref{thm:wvp-implies-swvp}, which says that WVP implies SWVP. (The converse is obvious.) We work in G\"odel--Bernays set theory with the Axiom of Choice for sets. Assume that Semi-Weak Vop\v{e}nka's Principle fails and let $\langle G_\alpha : \alpha \in \Ord\rangle$ be a counterexample, meaning that for every pair of ordinals $\alpha \le \alpha'$ there is a homomorphism from $G_{\alpha'}$ to $G_{\alpha}$ and for every pair of ordinals $\alpha < \alpha'$ there is no homomorphism from $G_\alpha$ to $G_{\alpha'}$. For every ordinal $\beta$ we define the product graph
\[H_\beta = \prod_{\alpha < \beta} G_{\alpha+1}.\]

Note that the sequence $\langle H_\beta : \beta\in \Ord\rangle$ is also a counterexample to SWVP: for every pair of ordinals $\beta \le \beta'$ there is a restriction homomorphism from $H_{\beta'}$ to $H_{\beta}$ given by $f \mapsto f \restriction \beta$, and for every pair of ordinals $\beta < \beta'$ there is no homomorphism from $H_\beta$ to $H_{\beta'}$ because if there were, then as in the proof of Proposition \ref{prop:V-kappa-satisfies-SWVP} we could compose homomorphisms
\[ G_\beta \to H_\beta \to H_{\beta'} \to G_{\beta+1}\]
to obtain a homomorphism from $G_\beta$ to $G_{\beta+1}$, which is a contradiction.

The sequence of product graphs $\langle H_\beta : \beta\in \Ord\rangle$ has the advantage that its ``backward'' homomorphisms are given simply by restriction.  We will use this sequence to build a counterexample to Weak Vop\v{e}nka's Principle.

Let $\Lambda$ be the class of all limit ordinals $\lambda$ such that $G_{\alpha+1} \in V_{\lambda}$ for all $\alpha < \lambda$, which is a proper class because it is the class of closure points of the function $\alpha \mapsto \rank(G_{\alpha+1})$. For every pair of ordinals $\lambda \le \lambda'$ in $\Lambda$ we define the function $h_{\lambda'\lambda} : V_{\lambda'+1} \to V_{\lambda+1}$ by
\[h_{\lambda'\lambda} (x) = x \cap V_{\lambda}.\]

We will define some structure (constants and relations) on the sets $V_{\lambda+1}$ ($\lambda \in \Lambda$) that is preserved by these functions $h_{\lambda'\lambda}$ and not by any other functions. To ensure that the structure is not preserved by any other functions, we will need to code our counterexample to SWVP into the structures in some way, because we cannot hope to define a counterexample to WVP from scratch. The assumption $\lambda \in \Lambda$ will ensure that we have enough room to code the sequence of product graphs $\langle H_\beta : \beta \le \lambda\rangle$ into our structure on $V_{\lambda+1}$ in a nice way.

Let $\mathcal{L}$ be the language with a constant symbol $c$ and ternary relation symbols $R$, $S$, and $T$. For every ordinal $\lambda \in \Lambda$ we define a corresponding $\mathcal{L}$-structure
\[\mathcal{M}_\lambda = \langle V_{\lambda+1}, c^{\mathcal{M}_\lambda}, R^{\mathcal{M}_\lambda}, S^{\mathcal{M}_\lambda}, T^{\mathcal{M}_\lambda}\rangle\]
where $c^{\mathcal{M}_\lambda} = \lambda$ and the interpretations of $R$, $S$, and $T$ are defined as follows.
\begin{align*}
 R^{\mathcal{M}_\lambda}(\beta, x,y) &\iff \big(\beta = \rank(x) \text{ and } x \in y\big) \text{ or } \beta = \lambda\\
 S^{\mathcal{M}_\lambda}(\beta, x,y) &\iff \big(\beta = \rank(x) \text{ and } x \notin y\big) \text{ or } \beta = \lambda\\
 T^{\mathcal{M}_\lambda}(\beta, x,y) &\iff \text{$x$ is adjacent to $y$ in $H_\beta$.}
\end{align*}

Here we take $\beta$ to be a von Neumann ordinal, so $\beta \in V_{\lambda+1}$ means $\beta \le \lambda$. In the definition of $T^{\mathcal{M}_\lambda}$ we take $x$ and $y$ to be vertices of $H_\beta$. Because $\lambda \in \Lambda$, a straightforward calculation using the definition of $\Lambda$ and the set-theoretic definition of products shows that for every ordinal $\beta \le \lambda$, every vertex of $H_\beta$ has rank at most $\lambda$ and is therefore an element of $V_{\lambda+1}$.

First we show that the functions $h_{\lambda'\lambda}$ are indeed homomorphisms of these structures.

\begin{claim}
 For every pair of ordinals $\lambda \le \lambda'$ in $\Lambda$, $h_{\lambda'\lambda}$ is a homomorphism $\mathcal{M}_{\lambda'} \to \mathcal{M}_\lambda$.
\end{claim}
\begin{proof}
 Let us abbreviate $h_{\lambda'\lambda}$ as $h$. Note that for every ordinal $\beta \le \lambda'$ we have $h(\beta) = \beta \cap V_\lambda = \min(\beta, \lambda)$. In particular we have $h(\lambda') = \lambda$, meaning that $h$ preserves $c$.
 
 To show preservation of $R$, let $\langle \beta, x,y \rangle \in R^{\mathcal{M}_{\lambda'}}$. First we consider the case where $\beta < \lambda$. Then because $\beta \ne \lambda'$ the triple $\langle \beta, x,y \rangle \in R^{\mathcal{M}_{\lambda'}}$ must satisfy the first clause of the definition of $R^{\mathcal{M}_{\lambda'}}$, meaning that $\beta = \rank(x)$ and $x \in y$. Because $\beta < \lambda$ we have $h(\beta) = \beta$, and because $\rank(x) < \lambda$ we have $h(x) = x$. Because $\rank(x) < \lambda$ and $x \in y$, we have $x \in h(y)$. Therefore we can write $h(\beta) = \rank(h(x))$ and $h(x) \in h(y)$, so the triple $\langle h(\beta), h(x),h(y) \rangle$ satisfies the first clause of the definition of $R^{\mathcal{M}_{\lambda}}$.
 
 Next we consider the case where $\beta \ge \lambda$. Then $h(\beta) = \lambda$, so $\langle h(\beta), h(x),h(y) \rangle$ satisfies the second clause of the definition of $R^{\mathcal{M}_{\lambda}}$. In both cases $\langle h(\beta), h(x),h(y) \rangle$ is in $R^{\mathcal{M}_{\lambda}}$.
  
 Preservation of $S$ is proved by a similar argument. The only difference is that instead of $x \in y \implies x \in h(y)$ we use $x \notin y \implies x \notin h(y)$.
  
 To establish preservation of $T$, we simply note that for all $\beta \le \lambda'$ and all $x$ in the product graph $H_\beta$, if $\beta \le \lambda$ then $h(x) = x \cap V_\lambda = x$ and if $\beta > \lambda$ then $h(x) = x \cap V_\lambda = x \restriction \lambda$. This is a straightforward consequence of the fact that $\lambda \in \Lambda$, the definition of $\Lambda$, and the set-theoretic definition of products. Then in the case $\beta > \lambda$ we use the fact that the restriction map $x \mapsto x \restriction \lambda$ is a homomorphism from $H_\beta$ to $H_\lambda$.
\end{proof}

Next we show that there are no ``forward'' homomorphisms among these structures.

\begin{claim}\label{claim:no-forward-homs}
 For every pair of ordinals $\lambda < \lambda'$ in $\Lambda$, there is no homomorphism $\mathcal{M}_{\lambda}\to\mathcal{M}_{\lambda'}$.
\end{claim}
\begin{proof}
 Assume toward a contradiction that there is such a homomorphism $h$.  Preservation of $c$ implies $h(\lambda) = \lambda'$, and in conjunction with preservation of $T$ this implies that $h \restriction H_\lambda$ is a homomorphism from $H_\lambda$ to $H_{\lambda'}$, contradicting the fact that the sequence of product graphs $\langle H_\beta : \beta \in \Ord\rangle$ is a counterexample to SWVP.
\end{proof}

Finally, we show that the homomorphisms $h_{\lambda'\lambda}$ are the only ``backward'' homomorphisms.

\begin{claim}
 For every pair of ordinals $\lambda \le \lambda'$ in $\Lambda$, every homomorphism from $\mathcal{M}_{\lambda'}$ to $\mathcal{M}_\lambda$ is equal to $h_{\lambda'\lambda}$.
\end{claim}
\begin{proof}
 Let $h$ be a homomorphism from $\mathcal{M}_{\lambda'}$ to $\mathcal{M}_\lambda$. We want to show that $h(y) = y \cap V_\lambda$ for all $y \in V_{\lambda'+1}$.  First, we note that $h$ cannot map any ordinal to a larger ordinal: if $h(\beta) = \beta'$ where $\beta < \beta'$, then preservation of $T$ implies that $h \restriction H_\beta$ is a homomorphism from $H_\beta$ to $H_{\beta'}$ and we get a contradiction as in the proof of Claim \ref{claim:no-forward-homs}.
 
 Now let $x \in V_\lambda$ and $y \in V_{\lambda'+1}$.
 If $x \in y$, then applying preservation of $R$ to the triple $\langle \beta, x, y\rangle \in R^{\mathcal{M}_{\lambda'}}$ where $\beta = \rank(x)$ shows that $h(\beta) = \rank(h(x))$ and $h(x) \in h(y)$. This is because $h$ cannot map $\beta$ to the larger ordinal $\lambda$, so we must end up in the first clause of the definition of $R^{\mathcal{M}_\lambda}$ rather than the second clause. Moreover, $h$ cannot increase the rank of $x$ because $h(\beta) \le \beta$.
 If $x \notin y$, a similar argument using preservation of $S$ instead of $R$ shows that $h(x) \notin h(y)$ and again $h$ cannot increase the rank of $x$. Therefore we have:
 \begin{enumerate}
  \item\label{item:rank-non-increasing} $\rank(h(x)) \le \rank(x)$ for all $x \in V_\lambda$, and 
  \item\label{item:preserves-in-and-notin} $h(x) \in h(y) \iff x \in y$, for all $x \in V_\lambda$ and $y \in V_{\lambda'+1}$.
 \end{enumerate}
 
 When applied to the case where $y$ is also an element of $V_{\lambda}$, statements \eqref{item:rank-non-increasing} and \eqref{item:preserves-in-and-notin} show that $h \restriction V_\lambda$ is a function from $V_\lambda$ to $V_\lambda$ that preserves $\in$ and $\notin$ and does not increase rank.  It is well-known that any such function must be the identity,\footnote{If not, let $x \in V_\lambda$ be a set of minimal rank that is moved by $h$. Then we obtain a contradiction using the fact that both $x$ and $h(x)$ are determined by which sets of lesser rank they do and do not contain.}
 so statement \eqref{item:preserves-in-and-notin} becomes
 \begin{quote}
  $x \in h(y) \iff x \in y$, for all $x \in V_\lambda$ and $y \in V_{\lambda'+1}$.
 \end{quote}
 
 Now let $y \in V_{\lambda'+1}$.
 Because $h(y) \in V_{\lambda+1}$, meaning $h(y) \subset V_\lambda$, the above statement implies that $h(y) = y \cap V_{\lambda}$, which agrees with $h_{\lambda'\lambda}$.
\end{proof}

We showed that the only homomorphisms among the structures $\mathcal{M}_\lambda$ for $\lambda \in \Lambda$ are the homomorphisms $h_{\lambda'\lambda}$ for $\lambda \le \lambda'$. We may enumerate our index class $\Lambda$ in increasing order as $\langle \lambda_\alpha : \alpha \in \Ord\rangle$ to obtain an Ord-sequence of structures $\langle \mathcal{M}_{\lambda_\alpha} : \alpha \in \Ord\rangle$. This is a counterexample to WVP except that the structures are $\mathcal{L}$-structures instead of graphs. However, the category of $\mathcal{L}$-structures fully embeds into the category of graphs by Hedrl\'{\i}n and Pultr \cite{HedPulFullEmbeddings}, so we can replace each structure $\mathcal{M}_{\lambda_\alpha}$ by its image under such an embedding to obtain an $\Ord$-sequence of graphs that is a counterexample to WVP.

\bibliographystyle{plain}
\bibliography{WVP-vs-VP}

\end{document}